\documentclass[11pt,reqno]{amsart}
\usepackage{geometry}
\usepackage{amsmath,amssymb,amsthm}
\usepackage{amscd}
\usepackage[colorlinks=true,citecolor=blue,linkcolor=blue,urlcolor=blue]{hyperref}
\hypersetup{
  pdftitle={A characterization of triangulated structures induced by premutation triples},
  pdfauthor={Panyue Zhou},
  pdfsubject={Extriangulated categories and mutation},
  pdfkeywords={extriangulated category, premutation triple, pretriangulated category, triangulated quotient, mutation}
}

\numberwithin{equation}{section}

\newtheorem{theorem}{Theorem}[section]
\newtheorem{proposition}[theorem]{Proposition}
\newtheorem{lemma}[theorem]{Lemma}
\newtheorem{corollary}[theorem]{Corollary}
\theoremstyle{definition}
\newtheorem{definition}[theorem]{Definition}

\newtheorem{axiom}[theorem]{Axiom}
\newtheorem{remark}[theorem]{Remark}
\newtheorem{example}[theorem]{Example}

\newcommand{\C}{\mathcal C}
\newcommand{\Ssub}{\mathcal S}
\newcommand{\Zsub}{\mathcal Z}
\newcommand{\Vsub}{\mathcal V}
\newcommand{\Isub}{\mathcal I}
\newcommand{\Uti}{\widetilde{\mathcal U}}
\newcommand{\Tti}{\widetilde{\mathcal T}}
\newcommand{\E}{\mathbb E}
\newcommand{\s}{\mathfrak s}
\DeclareMathOperator{\add}{\mathsf{add}\hspace{-0.8pt}}
\DeclareMathOperator{\moduleCategory}{\mathsf{mod}}

\renewcommand{\mod}{\moduleCategory}
\DeclareMathOperator{\Cone}{Cone\hspace{0.5pt}}
\DeclareMathOperator{\CoCone}{CoCone\hspace{0.5pt}}
\newcommand{\Id}{\mathrm{Id}}
\DeclareMathOperator{\Hom}{Hom}
\newcommand{\underlineZ}{\underline{\mathcal Z}}
\newcommand{\ul}[1]{\underline{#1}}
\newcommand{\xdash}[1]{\overset{#1}{\dashrightarrow}}

\title[Triangulated structures induced by premutation triples]{A characterization of triangulated structures induced\\[2mm] by premutation triples}

\author{Panyue Zhou}
\address{School of Mathematics and Statistics, Changsha University of Science and Technology, 410114 Changsha, Hunan, P. R. China}
\email{panyuezhou@163.com}

\makeatletter
\@namedef{subjclassname@2020}{\textup{2020} Mathematics Subject Classification}
\makeatother
\subjclass[2020]{16E30; 16G10; 16D90}
\keywords{extriangulated category; premutation triple; pretriangulated category; triangulated quotient; mutation}

\begin{document}

\begin{abstract}
Let $(\mathcal{S},\mathcal{Z},\mathcal{V})$ be a premutation triple in an extriangulated category. Set $\mathcal{I}=\mathcal{S}\cap\mathcal{Z}=\mathcal{Z}\cap\mathcal{V}$.
Iitsuka introduced premutation triples as a unified framework for various mutation constructions and showed that the quotient category $\mathcal{Z}/[\mathcal{I}]$ admits a canonical pretriangulated structure. In this paper, we prove that the compatibility condition $(\mathrm{MT4})$ is not only sufficient but also necessary for this structure to be triangulated. Therefore, $(\mathrm{MT4})$ precisely characterizes when Iitsuka's canonical pretriangulated structure becomes triangulated.
Furthermore, we prove that, for concentric twin cotorsion pairs, the compatibility condition used by Nakaoka in the construction of triangulated quotients is equivalent to $(\mathrm{MT4})$. Finally, we construct an explicit reducible triple which is not a mutation triple. This gives a negative answer to a question posed by Iitsuka.
\end{abstract}

\maketitle

\section{Introduction}

Mutation is a fundamental method in representation theory for constructing new objects and quotient categories. Iyama and Yoshino \cite{IyamaYoshino} developed mutation with respect to rigid subcategories, including cluster tilting mutation. Coelho  Sim\~oes  and Pauksztello \cite{SimoesPauksztello} studied mutation of orthogonal collections in connection with simple-minded systems. Zhou and Zhu \cite{ZhouZhu}, and Zheng, Cao and Wei \cite{ZhengCaoWei}, obtained related quotient constructions from functorially finite subcategories. Nakaoka and Palu \cite{NakaokaPalu} introduced extriangulated categories as a common framework for exact and triangulated categories.

Iitsuka \cite{Iitsuka} introduced premutation triples as a common generalization of Frobenius extriangulated categories, rigid mutation pairs, orthogonal mutation pairs and constructions arising from concentric twin cotorsion pairs. Let
$$
  (\Ssub,\Zsub,\Vsub),\hspace{2mm}
  \Isub=\Ssub\cap\Zsub=\Zsub\cap\Vsub,
$$
be a premutation triple. Iitsuka \cite[Definition~3.14]{Iitsuka} constructed mutation functors
$$
  \Sigma,\Omega\colon \Zsub/[\Isub]\longrightarrow\Zsub/[\Isub]
$$
and proved that they endow the quotient with a canonical pretriangulated structure
$$
  (\Zsub/[\Isub],\Sigma,\Omega,\nabla,\Delta),
$$
see \cite[Theorem~3.37]{Iitsuka}. This construction provides a unified categorical framework for the mutation settings mentioned above. Iitsuka further introduced the compatibility condition $(\mathrm{MT4})'$ in \cite[Condition~4.6]{Iitsuka}. By \cite[Remark~4.9(1)]{Iitsuka}, this condition is equivalent to $(\mathrm{MT4})$. By \cite[Theorem~4.8]{Iitsuka}, these equivalent conditions are sufficient for the canonical pretriangulated category to be triangulated.

It is natural to ask whether the converse holds. Notice that $(\mathrm{MT4})$ is defined by relative extensions in the ambient extriangulated category, while triangulatedness is a property of the quotient category. In our main result, we prove that the converse is true.

\begin{theorem}\label{thm:intro-main}
Let $(\Ssub,\Zsub,\Vsub)$ be a premutation triple in an extriangulated category $(\C,\E,\s)$, and put
$$
  \Isub=\Ssub\cap\Zsub=\Zsub\cap\Vsub.
$$
Then the following statements are equivalent.
\begin{enumerate}
\item The canonical pretriangulated category
$$
  (\Zsub/[\Isub],\Sigma,\Omega,\nabla,\Delta)
$$
constructed in \cite[Theorem~3.37]{Iitsuka} is a triangulated category.
\vspace{2mm}

\item The premutation triple satisfies Iitsuka's condition $(\mathrm{MT4})'$ of \cite[Condition~4.6]{Iitsuka}.
\vspace{2mm}

\item The premutation triple satisfies $(\mathrm{MT4})$ of \cite[Remark~4.9(1)]{Iitsuka}.
\end{enumerate}
\vspace{2mm}

If $\ul{\mathcal X}$ denotes the image of a subcategory $\mathcal X$ in $\C/[\Isub]$, then the equivalent condition in {\rm (3)} is
$$
 (\Zsub/[\Isub],\Sigma,\Omega,\nabla,\Delta)
 \text{ is triangulated}
 \quad\Longleftrightarrow\quad
 \ul{\Zsub\langle1\rangle}^{-}=\ul{\Zsub\langle-1\rangle}^{+}.
$$
We recall the relative shifts $\langle1\rangle$, $\langle-1\rangle$ and the functors $(\,\cdot\,)^{-}$, $(\,\cdot\,)^{+}$ in {\rm Section~\ref{sec:prelim}}.
\end{theorem}

Iitsuka proves $(2)\Rightarrow(1)$ in \cite[Theorem~4.8]{Iitsuka} and $(2)\Leftrightarrow(3)$ in \cite[Remark~4.9(1)]{Iitsuka}. We therefore only need to prove $(1)\Rightarrow(2)$.
\vspace{2mm}

By \cite[Lemma~3.16 and Notation~3.17]{Iitsuka}, $\Sigma$ is left adjoint to $\Omega$, with unit $\alpha$ and counit $\beta$. If the canonical pretriangulated category is triangulated, then $\Sigma$ is an equivalence. Hence $\alpha$ and $\beta$ are isomorphisms. We use their explicit descriptions in \cite[Remark~3.18]{Iitsuka} to prove the two inclusions in $(\mathrm{MT4})'$.
\vspace{1mm}

We give two applications. First, let $((\Ssub,\mathcal T),(\mathcal U,\Vsub))$ be a concentric twin cotorsion pair. Iitsuka \cite{Iitsuka} obtained two inclusions between quotient categories from Nakaoka's construction. He showed that these two inclusions are equivalent to $(\mathrm{MT4}^{-})$ for the associated premutation triple. We prove that $(\mathrm{MT4}^{-})$ is equivalent to $(\mathrm{MT4})$ in this case. Hence the two inclusions, $(\mathrm{MT4}^{-})$ and $(\mathrm{MT4})$ are all equivalent.
Second, Iitsuka \cite{Iitsuka} asked whether every reducible triple is a mutation triple. We show that the answer is negative. We first give an explicit example in the module category of a path algebra of type $A_4$. We then give another example over a finite-dimensional bound quiver algebra. In this example, the associated quotient is nonzero.

Throughout this paper, $k$ is a field and $\C$ is a skeletally small additive category. All subcategories are additive, full and closed under isomorphisms. We do not assume that they are closed under direct summands unless this is stated. For a subcategory $\mathcal X$, we denote by $\add\mathcal X$ the smallest subcategory containing $\mathcal X$ and closed under direct summands. For an object $X$, we write $\add X$ instead of $\add\{X\}$. Iitsuka uses these conventions in the opening paragraph of \cite[Section~2]{Iitsuka}.
\vspace{1mm}

Unless explicitly stated otherwise, we use Iitsuka's terminology and notation for relative extriangulated structures, relative shifts, premutation triples, mutation triples, reducible triples, mutation functors, the functors $(\,\cdot\,)^{-}$ and $(\,\cdot\,)^{+}$, and the canonical right, left and pretriangulated structures \cite[Sections~2--5 and Appendix~A]{Iitsuka}. We give precise references when each construction first occurs below.

We organize the paper as follows. In Section~\ref{sec:prelim}, we recall the necessary facts on premutation triples. In Section~\ref{sec:converse}, we prove two converse lemmas and Theorem~\ref{thm:intro-main}. In Section~\ref{sec:applications}, we apply our characterization to concentric twin cotorsion pairs and reducible triples. In Section~\ref{sec:relations}, we compare our result with earlier criteria for quotient categories.

\section{Preliminaries}\label{sec:prelim}

We retain the standing conventions stated in the introduction. In particular, $k$ is a field, $\C$ is skeletally small, and subcategories are additive, full and closed under isomorphisms without an implicit direct summand assumption. From now on $(\C,\E,\s)$ is an extriangulated category. We use the standard $\E$-triangle notation of Nakaoka and Palu \cite[Definition~2.19]{NakaokaPalu}. Thus, if $\delta\in\E(C,A)$ is realized by $A\to B\to C$, we write
$$
  A\longrightarrow B\longrightarrow C\xdash{\delta}.
$$
The extension label is placed over the middle of the dashed arrow, as in \cite[Definition~2.19]{NakaokaPalu}.

For subcategories $\mathcal X,\mathcal Y\subseteq\C$, we use Iitsuka's cone and cocone conventions \cite[Definition~2.8(3)]{Iitsuka}. Explicitly,
$$
  \Cone_{\E}(\mathcal X,\mathcal Y)
  =\{Z\in\C\mid{}\text{there is an $\s$-triangle }X\longrightarrow Y\longrightarrow Z\dashrightarrow,
                   X\in\mathcal X,\ Y\in\mathcal Y\},
$$
and
$$
  \CoCone_{\E}(\mathcal X,\mathcal Y)
  =\{Z\in\C\mid{}\text{there is an $\s$-triangle }Z\longrightarrow X\longrightarrow Y\dashrightarrow,
                   X\in\mathcal X,\ Y\in\mathcal Y\}.
$$
When the extension bifunctor is clear, we omit the subscript. Following Nakaoka and Palu \cite[Definition~2.17]{NakaokaPalu}, we call a subcategory $\mathcal D$ \emph{extension closed} if the middle term of every $\s$-triangle with first and third terms in $\mathcal D$ also belongs to $\mathcal D$.

Following Beligiannis--Reiten \cite[II.1]{BeligiannisReiten}, as recalled in \cite[Definition~2.37]{Iitsuka}, a pretriangulated category has a right triangulated structure and a left triangulated structure whose shift functors form an adjoint pair and satisfy the compatibility axioms. In this paper, the statement that Iitsuka's canonical pretriangulated category is triangulated always refers to this fixed structure. In particular, its suspension functor $\Sigma$ is an equivalence and its right adjoint $\Omega$ is a quasi-inverse.

\subsection{Relative extriangulated structures and approximations}

\begin{definition}\label{def:relative-structures}
Let $\Isub$ be a subcategory of $\C$. Following Herschend, Liu and Nakaoka \cite[Definition~3.18 and Proposition~3.19]{HerschendLiuNakaoka}, as recalled by Iitsuka \cite[Example~2.17]{Iitsuka}, we define closed subfunctors
\begin{align*}
\E_{\Isub}(C,A)
  &=\{\delta\in\E(C,A)\mid
      \delta\circ -:\C(\Isub,C)\to\E(\Isub,A)\text{ is zero}\},\\
\E^{\Isub}(C,A)
  &=\{\delta\in\E(C,A)\mid
      -\circ\delta:\C(A,\Isub)\to\E(C,\Isub)\text{ is zero}\}.
\end{align*}
For $\delta\in\E(C,A)$, $a:A\to A'$ and $c:C'\to C$, we write
$$
  a\delta=a_{*}\delta\in\E(C,A'),
  \qquad
  \delta c=c^{*}\delta\in\E(C',A).
$$
Thus a product such as $\lambda_Zfh^U$ means that one first pulls back $\lambda_Z$ along $f$ and then along $h^U$. This is the convention of \cite[Remark~2.2(4)--(5)]{Iitsuka}.
These subfunctors determine relative extriangulated structures
$$
  (\C,\E_{\Isub},\s_{\Isub})
 \quad\text{and}\quad
  (\C,\E^{\Isub},\s^{\Isub}).
$$
Following Iitsuka \cite[Example~2.17(3)]{Iitsuka}, we also write
$$
  \E^{\Isub}_{\Isub}=\E^{\Isub}\cap\E_{\Isub}
$$
and denote the induced intersection relative structure by
$$
  (\C,\E^{\Isub}_{\Isub},\s^{\Isub}_{\Isub}).
$$
Here a closed subfunctor means an additive subfunctor for which the restricted realization is again an extriangulated structure, see Herschend, Liu and Nakaoka \cite[Proposition~3.16]{HerschendLiuNakaoka} and Iitsuka \cite[Proposition~2.16]{Iitsuka}. In particular, all three relative structures satisfy $(\mathrm{ET4})$ and $(\mathrm{ET4})^{\mathrm{op}}$.
\end{definition}

We use the diagrammatic forms of the octahedral axioms given by Nakaoka and Palu \cite[Definition~2.12 and Remark~2.22]{NakaokaPalu}. We record them separately to fix the orientation of the diagrams used later.

\begin{axiom}[The octahedral axiom $(\mathrm{ET4})$]\label{axiom:et4}
Suppose that
$$
 A\xrightarrow{f}B\xrightarrow{f'}D\xdash{\delta}
 \quad\text{and}\quad
 B\xrightarrow{g}C\xrightarrow{g'}F\xdash{\delta'}
$$
are $\s$-triangles. We require an object $E$ and a commutative diagram
$$
\begin{CD}
A @>{f}>> B @>{f'}>> D\\
@V{1_A}VV @V{g}VV @V{d}VV\\
A @>{h}>> C @>{h'}>> E\\
@. @V{g'}VV @V{e}VV\\
@. F @= F
\end{CD}
$$
such that
$$
 A\xrightarrow{h}C\xrightarrow{h'}E\xdash{\delta''}
 \quad\text{and}\quad
 D\xrightarrow{d}E\xrightarrow{e}F\xdash{f'_{*}\delta'}
$$
are $\s$-triangles. The extension identities are
$$
 d^{*}\delta''=\delta,
 \qquad
 f_{*}\delta''=e^{*}\delta'.
$$
\end{axiom}

\begin{axiom}[The dual octahedral axiom $(\mathrm{ET4})^{\mathrm{op}}$]\label{axiom:et4op}
If
$$
 D\xrightarrow{f'}A\xrightarrow{f}B\xdash{\delta}
 \quad\text{and}\quad
 F\xrightarrow{g'}B\xrightarrow{g}C\xdash{\delta'}
$$
are $\s$-triangles, then we require an object $E$ and a commutative diagram
$$
\begin{CD}
D @>{d}>> E @>{e}>> F\\
@V{1_D}VV @V{h'}VV @V{g'}VV\\
D @>{f'}>> A @>{f}>> B\\
@. @V{h}VV @V{g}VV\\
@. C @= C
\end{CD}
$$
for which
$$
 E\xrightarrow{h'}A\xrightarrow{h}C\xdash{\delta''}
 \quad\text{and}\quad
 D\xrightarrow{d}E\xrightarrow{e}F\xdash{(g')^{*}\delta}
$$
are $\s$-triangles. The corresponding identities are
$$
 \delta'=e_{*}\delta'',
 \qquad
 d_{*}\delta=g^{*}\delta''.
$$
\end{axiom}

\begin{definition}\label{def:approximations}
A morphism $I_X\to X$ with $I_X\in\Isub$ is a right $\Isub$-approximation if every map from an object of $\Isub$ to $X$ factors through it. Left approximations are defined dually. A morphism $f:X\to Y$ is $\Isub$-monic if every morphism $X\to I$ with $I\in\Isub$ factors through $f$, and $g:Y\to Z$ is $\Isub$-epic if every morphism $I\to Z$ with $I\in\Isub$ factors through $g$. Iitsuka \cite[Lemma~2.20]{Iitsuka} proves that, for an $\s$-triangle $X\xrightarrow{f}Y\xrightarrow{g}Z\xdash{\delta}$, one has $\delta\in\E^{\Isub}(Z,X)$ exactly when $f$ is $\Isub$-monic, and $\delta\in\E_{\Isub}(Z,X)$ exactly when $g$ is $\Isub$-epic. Following Iitsuka \cite[Definition~2.21]{Iitsuka}, we assume that a subcategory used in the notion of strong functorial finiteness is closed under direct summands. We say that $\Isub$ is \emph{strongly contravariantly finite} in a subcategory $\mathcal X$ if every $X\in\mathcal X$ admits a right $\Isub$-approximation which is a deflation. We define strong covariant finiteness dually, and \emph{strong functorial finiteness} means both.
\end{definition}

If $\Isub$ is strongly functorially finite in $\Zsub$, then for every $Z\in\Zsub$ one fixes relative triangles
\begin{equation}\label{eq:cosyzygy}
  Z\xrightarrow{i^Z} I^Z\longrightarrow Z\langle1\rangle
  \xdash{\lambda^Z},
  ~~\mbox{where}~I^Z\in\Isub,
\end{equation}
and
\begin{equation*}
  Z\langle-1\rangle\longrightarrow I_Z\xrightarrow{p_Z}Z
  \xdash{\lambda_Z},
  ~~\mbox{where}~ I_Z\in\Isub.
\end{equation*}
Under the premutation hypotheses recalled below, these approximation triangles belong to the intersection relative structure whenever both relative conditions are needed. This follows from \cite[Lemma~2.24, Lemma~2.32 and Remark~3.2(3)--(4)]{Iitsuka} and the dual statements. They induce additive functors $\langle1\rangle$ and $\langle-1\rangle$ after passing to the quotient by $[\Isub]$.
Following \cite[Notation~2.26]{Iitsuka}, for a subcategory $\mathcal X$ we write $\mathcal X\langle1\rangle$ and $\mathcal X\langle-1\rangle$ for the corresponding cone and cocone subcategories. Equivalently, these are represented in the quotient by the objects $X\langle1\rangle$ and $X\langle-1\rangle$ with $X\in\mathcal X$.

\begin{definition}\label{def:ideal-quotient}
For a subcategory $\mathcal X\subseteq\C$, we write $[\mathcal X](A,B)$ for the subgroup of morphisms $A\to B$ factoring through an object of $\mathcal X$. We denote the additive quotient by $\C/[\mathcal X]$. If $\mathcal X$ is a subcategory of $\C$, we write $\ul{\mathcal X}$ for its full image in $\C/[\Isub]$. Thus $\ul{\Zsub}=\Zsub/[\Isub]$. We use underlines whenever a statement concerns objects, subcategories or functors in the quotient. This convention is important for Conditions $(\mathrm{MT4})'$ and $(\mathrm{MT4})$ below.
\end{definition}

We shall use the following stable lifting statement exactly in the form proved by Iitsuka.

\begin{lemma}\cite[Corollary~2.31]{Iitsuka}\label{lem:stable-lift}
Assume that $\Isub\subseteq\Zsub$ is strongly functorially finite. If a morphism $f:X\to Y$ in $\Zsub$ becomes an isomorphism in $\C/[\Isub]$, then there exist $I,J\in\Isub$ and an isomorphism in $\C$
$$
  \begin{pmatrix} f & j\\ i & k\end{pmatrix}:
  X\oplus J\xrightarrow{\sim}Y\oplus I.
$$
\end{lemma}

\subsection{Premutation triples}
We now fix the data used throughout the paper. The following formulation is equivalent to \cite[Conditions~3.1, 3.9 and Definition~3.11]{Iitsuka}.

\begin{definition}\label{def:premutation}
A triple $(\Ssub,\Zsub,\Vsub)$ of subcategories of $\C$ is a \emph{premutation triple} if the following conditions hold.
\begin{enumerate}
\item[(MT1)]
$$
  \Ssub\cap\Zsub=\Zsub\cap\Vsub=:\Isub,
$$
and $\Isub$ is strongly functorially finite in $\Zsub$.

\item[(MT2)] The relative orthogonality conditions
$$
  \E^{\Isub}(\Ssub,\Zsub)=0,
  \hspace{2mm}
  \E_{\Isub}(\Ssub,\Zsub\langle-1\rangle)=0,
$$
and
$$
  \E_{\Isub}(\Zsub,\Vsub)=0,
 \hspace{2mm}
  \E^{\Isub}(\Zsub\langle1\rangle,\Vsub)=0
$$
hold.

\item[(MT3)] Put
$$
  \Uti=\CoCone_{\E^{\Isub}}(\Zsub,\Ssub),
\hspace{2mm}
  \Tti=\Cone_{\E_{\Isub}}(\Vsub,\Zsub).
$$
Then
$$
  \Cone_{\E^{\Isub}}(\Zsub,\Zsub)\subseteq\Uti,
\hspace{2mm}
  \CoCone_{\E_{\Isub}}(\Zsub,\Zsub)\subseteq\Tti,
$$
$\Ssub$ and $\Zsub$ are extension closed in $(\C,\E^{\Isub},\s^{\Isub})$, and $\Zsub$ and $\Vsub$ are extension closed in $(\C,\E_{\Isub},\s_{\Isub})$.
\end{enumerate}
\end{definition}

Iitsuka \cite[Example~3.13]{Iitsuka} shows that premutation triples include the Frobenius case, rigid mutation pairs in the sense of Iyama and Yoshino \cite[Definition~2.5]{IyamaYoshino}, orthogonal mutation pairs related to simple-minded systems in the sense of Coelho Sim\~oes and Pauksztello \cite[Definition~3.2 and Setup~3.5]{SimoesPauksztello}, and triples obtained from concentric twin cotorsion pairs in the sense of Nakaoka \cite[Definition~3.3]{Nakaoka}.

By \cite[Lemma~3.6(1) and Remark~3.7(3)]{Iitsuka}, for $U\in\Uti$, Iitsuka constructs an $\s^{\Isub}$-triangle
\begin{equation}\label{eq:sigma-reflection}
  U\xrightarrow{h^U}\sigma U\longrightarrow S^U\dashrightarrow,
\quad\mbox{where}~~\sigma U\in\Zsub,~S^U\in\Ssub,
\end{equation}
and obtains a functor $\sigma:\Uti/[\Isub]\to\Zsub/[\Isub]$ which is left adjoint to the inclusion. Dually, for $T\in\Tti$ there is an $\s_{\Isub}$-triangle
\begin{equation}\label{eq:omega-coreflection}
  V_T\longrightarrow\omega T\xrightarrow{h_T}T\dashrightarrow,
  \quad\mbox{where}~ \omega T\in\Zsub,~ V_T\in\Vsub,
\end{equation}
and a right adjoint $\omega:\Tti/[\Isub]\to\Zsub/[\Isub]$ by \cite[Lemma~3.6(2)]{Iitsuka}.
In the rest of this paper, $h^U$ always denotes the fixed reflection morphism in \eqref{eq:sigma-reflection}, while $h_T$ denotes the fixed coreflection morphism in \eqref{eq:omega-coreflection}. For example, $h^{Z\langle1\rangle}$ and $h_{Z\langle-1\rangle}$ are obtained by substituting these objects for $U$ and $T$, respectively.

Iitsuka \cite[Definition~3.14]{Iitsuka} defines the right and left mutation functors by
\begin{equation*}
  \Sigma=\sigma\circ\langle1\rangle,
\hspace{2mm}
  \Omega=\omega\circ\langle-1\rangle.
\end{equation*}
We use the following structural result throughout the proof.

\begin{theorem}\label{thm:iitsuka-pretri}
Let $(\Ssub,\Zsub,\Vsub)$ be a premutation triple and put $\underlineZ=\Zsub/[\Isub]$. By \cite[Theorem~3.37]{Iitsuka},
$$
  (\underlineZ,\Sigma,\Omega,\nabla,\Delta)
$$
is a pretriangulated category. Here $\nabla$ and $\Delta$ are the classes of standard right and left triangles defined in \cite[Definition~3.26 and Definition~3.35]{Iitsuka}, respectively. Moreover, by \cite[Lemma~3.16]{Iitsuka}, $(\Sigma,\Omega)$ is an adjoint pair, with $\Sigma$ left adjoint to $\Omega$. Equivalently, there are natural bijections
$$
  \Hom_{\underlineZ}(\Sigma X,Y)
  \cong
  \Hom_{\underlineZ}(X,\Omega Y)
$$
for all $X,Y\in\underlineZ$.
\end{theorem}

For a morphism $a:X\to Y$ in $\Zsub$, the standard right triangle of $\ul{a}$ has the form
$$
 X\xrightarrow{\ul{a}}Y\longrightarrow\sigma C^a
 \longrightarrow\Sigma X,
$$
where $C^a$ is obtained from the fixed approximation triangle as in \cite[Notation~3.22 and Definition~3.26]{Iitsuka}. Dually, the standard left triangle of a morphism $b:Y\to Z$ has the form
$$
 \Omega Z\longrightarrow\omega C_b\longrightarrow Y
 \xrightarrow{\ul{b}}Z
$$
as in \cite[Notation~3.32 and Definition~3.35]{Iitsuka}. The classes of sequences isomorphic to these standard triangles are denoted by $\nabla$ and $\Delta$, respectively.

In \cite[Notation~3.17]{Iitsuka}, the unit and counit of the adjunction are denoted by
$$
  \alpha\colon\Id\longrightarrow\Omega\Sigma,
  \hspace{2mm}
  \beta\colon\Sigma\Omega\longrightarrow\Id.
$$
Thus $\alpha_Z$ and $\beta_Z$ are morphisms in $\ul{\Zsub}$. Whenever an ambient morphism is required, we choose representatives
$$
  \widetilde\alpha_Z\colon Z\longrightarrow\Omega\Sigma Z,
  \hspace{2mm}
  \widetilde\beta_Z\colon \Sigma\Omega Z\longrightarrow Z
$$
in $\Zsub$. The following observation records the precise form in which the explicit adjunction diagram of \cite[Remark~3.18]{Iitsuka} will be used.

\begin{lemma}\label{lem:roundtrip}
Let $Z\in\Zsub$.
\begin{enumerate}
\item Put $W=\Omega Z$ and $U=W\langle1\rangle$. For any representative $\widetilde\beta_Z$ of $\beta_Z$,
\begin{equation}\label{eq:beta-identity}
  \lambda_Z\,\widetilde\beta_Z\,h^U
  =-h_{Z\langle-1\rangle}\,\lambda^W.
\end{equation}
\item Put $W=\Sigma Z$ and $T=W\langle-1\rangle$. For any representative $\widetilde\alpha_Z$ of $\alpha_Z$,
\begin{equation}\label{eq:alpha-identity}
  h_T\,\widetilde\alpha_Z\,\lambda^Z
  =-\lambda_W\,h^{Z\langle1\rangle}.
\end{equation}
\end{enumerate}
Here compositions involving an extension are read in the standard extriangulated sense. In particular, $\lambda^W$ in {\rm part~(1)} is the cosyzygy extension of $W$, whereas $\lambda_W$ in
{\rm part~(2)} is the syzygy extension of $W$.
\end{lemma}

\begin{proof}
By \cite[Remark~3.18]{Iitsuka}, the adjunction correspondence is described by an explicit diagram, with the morphisms in $\ul{\Zsub}$ determined only modulo maps factoring through $\Isub$.
More explicitly, suppose that
$\ul{f}:X\to\Omega Y$ corresponds under the adjunction to
$\ul{f'}:\Sigma X\to Y$. Then ambient representatives $f$ and $f'$ can be
chosen so that
$$
  h_{Y\langle-1\rangle}f\lambda^X
  =-\lambda_Yf'h^{X\langle1\rangle}.
$$
Take $\ul{f}=1_{\Omega Z}$. We may choose $f=1_{\Omega Z}$ as its ambient
representative. The corresponding quotient morphism is
$\ul{f'}=\beta_Z$. Let $b_Z:\Sigma\Omega Z\to Z$ be the ambient
representative $f'$ supplied by the correspondence. Substituting
$X=\Omega Z$ and $Y=Z$ gives the formula below.
$$
  \lambda_Z b_Z h^U=-h_{Z\langle-1\rangle}\lambda^W.
$$
If $\widetilde\beta_Z$ is any other representative, then $\widetilde\beta_Z-b_Z$ factors through an object of $\Isub$. Since the fixed syzygy extension $\lambda_Z$ belongs to $\E_{\Isub}$, pullback along a morphism factoring through $\Isub$ is zero. Hence
$$
  \lambda_Z(\widetilde\beta_Z-b_Z)h^U=0,
$$
which proves \eqref{eq:beta-identity}.

For the second formula, take $1_{\Sigma Z}$ in the same adjunction correspondence. Its image under the inverse correspondence is $\alpha_Z$.
In the preceding formula, take $X=Z$, $Y=\Sigma Z$,
$\ul{f}=\alpha_Z$ and $\ul{f'}=1_{\Sigma Z}$. We may choose
$f'=1_{\Sigma Z}$ as an ambient representative. Let $a_Z$ be the ambient
representative $f$ supplied by the correspondence. This gives the formula
below.
$$
  h_Ta_Z\lambda^Z=-\lambda_Wh^{Z\langle1\rangle}.
$$
Any other representative differs from $a_Z$ by a morphism factoring through $\Isub$. The cosyzygy extension $\lambda^Z$ lies in $\E^{\Isub}$, so pushforward along such a morphism is zero. This gives \eqref{eq:alpha-identity}.
\end{proof}

Compared with Iitsuka's formula in \cite[Remark~3.18]{Iitsuka}, we also record the independence of representatives. This follows immediately from the defining relative extension conditions.

\subsection{The condition \texorpdfstring{$(\mathrm{MT4})$}{(MT4)}}

We next recall the functors $(\,\cdot\,)^{-}$ and $(\,\cdot\,)^{+}$ introduced by Iitsuka \cite[Definition~4.1, Notation~4.4, and Proposition~4.5]{Iitsuka}.

\begin{definition}[The functors $(\,\cdot\,)^{-}$ and $(\,\cdot\,)^{+}$]\label{def:corrections}
Define
\begin{equation}\label{eq:ut-minus-tt-plus}
  \Uti^{-}=\CoCone_{\E^{\Isub}}(\Isub,\Ssub),
  \hspace{2mm}
  \Tti^{+}=\Cone_{\E_{\Isub}}(\Vsub,\Isub).
\end{equation}
For $U\in\Uti$, combine \eqref{eq:sigma-reflection} with the syzygy triangle of $\sigma U$. Iitsuka's construction in \cite[Proposition~4.5]{Iitsuka} uses $(\mathrm{ET4})$ and yields an object $U^{-}$ fitting into an $\E^{\Isub}_{\Isub}$-triangle
\begin{equation*}
  \sigma U\langle-1\rangle\longrightarrow U^{-}\longrightarrow U
  \xdash{\chi_U}.
\end{equation*}
Its extension is
$$
  \chi_U=\lambda_{\sigma U}h^U.
$$
Dually, for $T\in\Tti$ there is an object $T^{+}$ in an $\E^{\Isub}_{\Isub}$-triangle
\begin{equation*}
  T\longrightarrow T^{+}\longrightarrow\omega T\langle1\rangle
  \xdash{\chi^T}.
\end{equation*}
Its extension is
$$
  \chi^T=h_T\lambda^{\omega T}.
$$
Iitsuka proves that these constructions induce adjoints to the inclusions $\Uti^{-}\subseteq\Uti$ and $\Tti^{+}\subseteq\Tti$ \cite[Proposition~4.5]{Iitsuka}.

More precisely, they induce functors
$$
 (\,\cdot\,)^{-}:\ul{\Uti}\longrightarrow\ul{\Uti^{-}},
 \quad
 (\,\cdot\,)^{+}:\ul{\Tti}\longrightarrow\ul{\Tti^{+}}.
$$
Fix the defining triangles once and for all. If $\mathcal X\subseteq\Uti$, let $\mathcal X^{-}$ be the full subcategory of $\C$ formed by the objects isomorphic in $\C$ to a chosen object $X^{-}$ with $X\in\mathcal X$. Define $\mathcal Y^{+}$ dually for $\mathcal Y\subseteq\Tti$. The notation $\ul{\mathcal X}^{-}$ denotes the full subcategory of $\ul{\Uti^{-}}$ formed by the images of the chosen objects, and $\ul{\mathcal Y}^{+}$ is defined dually. Thus $\mathcal X^{-}$ and $\mathcal Y^{+}$ are subcategories of $\C$, whereas $\ul{\mathcal X}^{-}$ and $\ul{\mathcal Y}^{+}$ are subcategories of the quotient.
\end{definition}

The inclusions $\ul{\Uti^{-}}\to\ul{\Uti}$ and
$\ul{\Tti^{+}}\to\ul{\Tti}$ are fully faithful. Consequently, the unit of
the first adjunction and the counit of the second are natural isomorphisms on
the smaller subcategories. In terms of essential images, this gives
$$
  \ul{\Uti^{-}}=\ul{\Uti}^{-},
  \quad
  \ul{\Tti^{+}}=\ul{\Tti}^{+}.
$$
We use these identifications below when we pass from membership of a chosen
object in $\Uti^{-}$ or $\Tti^{+}$ to an inclusion of subcategories in the
ideal quotient.

Iitsuka \cite[Condition~4.6]{Iitsuka} defines condition $(\mathrm{MT4})'$ by
\begin{equation}\label{eq:mt4prime}
  \ul{\Zsub\langle1\rangle}^{-}\subseteq\ul{\Tti}^{+},
  \quad
  \ul{\Zsub\langle-1\rangle}^{+}\subseteq\ul{\Uti}^{-}.
\end{equation}

By \cite[Remark~4.9(1)]{Iitsuka}, $(\mathrm{MT4})'$ is equivalent to
\begin{equation*}
  (\mathrm{MT4}):\quad
  \ul{\Zsub\langle1\rangle}^{-}=\ul{\Zsub\langle-1\rangle}^{+},
\end{equation*}
and Iitsuka \cite[Theorem~4.8]{Iitsuka} proves that these equivalent conditions imply that the pretriangulated structure of Theorem~\ref{thm:iitsuka-pretri} is triangulated. Following Iitsuka \cite[Definition~4.7]{Iitsuka}, we call a premutation triple satisfying $(\mathrm{MT4})'$ a \emph{mutation triple}. By \cite[Remark~4.9(1)]{Iitsuka}, this is equivalent to requiring $(\mathrm{MT4})$.

\section{The converse criterion}\label{sec:converse}

We now prove that the inclusions in \eqref{eq:mt4prime} are forced by invertibility of the counit and unit, respectively. The two arguments are dual, but we give both in detail because they take place in different relative extriangulated structures.

\subsection{The counit and the cocone condition}

\begin{lemma}\label{lem:counit-half}
Let $Z\in\Zsub$. If
$$
  \beta_Z:\Sigma\Omega Z\longrightarrow Z
$$
is an isomorphism in $\ul{\Zsub}$, then the following stronger statement holds in $\C$:
$$
  Z\langle-1\rangle^{+}\in\Uti^{-}.
$$
Consequently, if every component of the counit $\beta$ is invertible, then
$$
  \ul{\Zsub\langle-1\rangle}^{+}\subseteq\ul{\Uti}^{-}.
$$
\end{lemma}

\begin{proof}
Set
$$
  T=Z\langle-1\rangle,
  \quad W=\Omega Z,
  \quad U=W\langle1\rangle.
$$
By the chosen relative triangles and $(\mathrm{MT3})$, we have
$$
 T\in\CoCone_{\E_{\Isub}}(\Zsub,\Zsub)\subseteq\Tti,
 \quad
 U\in\Cone_{\E^{\Isub}}(\Zsub,\Zsub)\subseteq\Uti.
$$
Choose a representative
$$
  \widetilde\beta_Z:\Sigma\Omega Z\longrightarrow Z
$$
of $\beta_Z$. Since $\sigma U=\Sigma W=\Sigma\Omega Z$, the reflection triangle \eqref{eq:sigma-reflection} takes the form
\begin{equation}\label{eq:beta-reflection}
  U\xrightarrow{h^U}\Sigma\Omega Z\longrightarrow S^U\dashrightarrow,
  \qquad S^U\in\Ssub.
\end{equation}

Since $\Zsub$ is full, the isomorphism $\beta_Z$ in $\ul{\Zsub}$ is also an isomorphism between the same objects in $\C/[\Isub]$. By Lemma~\ref{lem:stable-lift}, applied to $\widetilde\beta_Z$, there are $I_0,I_1\in\Isub$ and an isomorphism in $\C$
\begin{equation*}
  B=
  \begin{pmatrix}
    \widetilde\beta_Z & a\\ b & c
  \end{pmatrix}\colon
  \Sigma\Omega Z\oplus I_0\xrightarrow{\sim} Z\oplus I_1.
\end{equation*}
Put
$$
  Y=Z\oplus I_1,
  \quad S'=S^U\oplus I_0.
$$
Taking the direct sum of \eqref{eq:beta-reflection} with a split triangle and transporting it along $B$, we obtain an $\s^{\Isub}$-triangle
\begin{equation}\label{eq:beta-first-input}
  U\xrightarrow{q}Y\xrightarrow{r}S'\dashrightarrow,
  ~~\mbox{where}~ S'\in\Ssub,
\end{equation}
where
\begin{equation}\label{eq:q-def}
  q=B\binom{h^U}{0}.
\end{equation}
In particular, the $Z$-component of $q$ is $\widetilde\beta_Zh^U$.

On the other hand, the fixed syzygy triangle of $Z$ is
$$
  T\longrightarrow I_Z\xrightarrow{p_Z}Z\xdash{\lambda_Z}.
$$
For objects of $\Zsub$, this triangle belongs to the intersection relative structure by the dual of \cite[Lemma~2.32]{Iitsuka}, using the orthogonality in (MT2). After taking the direct sum with the split triangle on $I_1$, we get an $\s^{\Isub}$-triangle
\begin{equation}\label{eq:beta-second-input}
  T\xrightarrow{i}I'\xrightarrow{p}Y
  \xdash{\widetilde\lambda_Z},
  ~~~\mbox{where}~
  I'=I_Z\oplus I_1\in\Isub,
\end{equation}
where
$$
  \widetilde\lambda_Z=(\lambda_Z,0).
$$

We may therefore apply $(\mathrm{ET4})^{\mathrm{op}}$ in the relative extriangulated category $(\C,\E^{\Isub},\s^{\Isub})$ to \eqref{eq:beta-first-input} and \eqref{eq:beta-second-input}. It gives an object $E_Z$ and $\s^{\Isub}$-triangles
\begin{equation}\label{eq:beta-output-one}
  T\xrightarrow{x}E_Z\xrightarrow{y}U
  \xdash{\chi_Z},
\end{equation}
and
\begin{equation}\label{eq:beta-output-two}
  E_Z\xrightarrow{c}I'\xrightarrow{rp}S'\dashrightarrow.
\end{equation}
We display the four conflations in this application of $(\mathrm{ET4})^{\mathrm{op}}$ as follows.
$$
\begin{CD}
T @>{x}>> E_Z @>{y}>> U\\
@V{1_T}VV @V{c}VV @V{q}VV\\
T @>{i}>> I' @>{p}>> Y\\
@. @V{rp}VV @V{r}VV\\
@. S' @= S'
\end{CD}
$$
This is the standard $(\mathrm{ET4})^{\mathrm{op}}$ diagram of Nakaoka--Palu \cite[Remark~2.22]{NakaokaPalu}. The top row is \eqref{eq:beta-output-one}. The middle row is \eqref{eq:beta-second-input}, the right column is \eqref{eq:beta-first-input}, and the middle column is \eqref{eq:beta-output-two}.

The extension compatibility in $(\mathrm{ET4})^{\mathrm{op}}$ says that the extension in \eqref{eq:beta-output-one} is the pullback of $\widetilde\lambda_Z$ along $q$. Hence
\begin{equation*}
  \chi_Z=q^{*}\widetilde\lambda_Z
  =(\widetilde\beta_Zh^U)^{*}\lambda_Z
  =\lambda_Z\,\widetilde\beta_Z\,h^U.
\end{equation*}
The second equality follows from the definition of $q$ in \eqref{eq:q-def} and the fact that the $I_1$-component of $\widetilde\lambda_Z$ is zero.

Now apply Lemma~\ref{lem:roundtrip}(i). Since $T=Z\langle-1\rangle$ and $W=\Omega Z$, we obtain
\begin{equation}\label{eq:chi-beta-identify}
  \chi_Z=-h_T\lambda^W.
\end{equation}
By the construction of $(\,\cdot\,)^{+}$ in \cite[Notation~4.4]{Iitsuka}, the extension $h_T\lambda^W$ is precisely the extension realized by the defining triangle
\begin{equation*}
  T\longrightarrow T^{+}\longrightarrow W\langle1\rangle=U\dashrightarrow.
\end{equation*}
Moreover, this defining extension lies in the intersection
$$
  \E^{\Isub}_{\Isub}(U,T)
  =\E^{\Isub}(U,T)\cap\E_{\Isub}(U,T).
$$
Equation \eqref{eq:chi-beta-identify} therefore shows that $\chi_Z$ lies in the intersection relative structure, although it was first obtained in $\E^{\Isub}$. Let
$$
  T\xrightarrow{s^T}T^+\longrightarrow U\xdash{h_T\lambda^W}
$$
be the triangle defining $T^+$. Since $\s$ is additive,
$$
 \s(-h_T\lambda^W)
 =\big[T\xrightarrow{-s^T}T^+\longrightarrow U\big].
$$
Thus \eqref{eq:beta-output-one} and this triangle are equivalent realizations of the same extension. In particular, their middle terms are isomorphic. Consequently,
\begin{equation}\label{eq:E-is-T-plus}
  E_Z\cong T^{+}=Z\langle-1\rangle^{+}.
\end{equation}

Finally, \eqref{eq:beta-output-two} has middle term $I'\in\Isub$ and third term $S'\in\Ssub$. By the definition \eqref{eq:ut-minus-tt-plus},
$$
  E_Z\in\CoCone_{\E^{\Isub}}(\Isub,\Ssub)=\Uti^{-}.
$$
Together with \eqref{eq:E-is-T-plus}, this gives
$$
  Z\langle-1\rangle^{+}\in\Uti^{-},
$$
as required.
For an arbitrary object $T'\in\Zsub\langle-1\rangle$, its image in the
quotient is isomorphic to the image of a chosen object
$Z\langle-1\rangle$. Applying the functor $(\,\cdot\,)^{+}$ gives
$\ul{(T')^{+}}\cong\ul{Z\langle-1\rangle^{+}}$. The object on the right is
the image of an object of $\Uti^{-}$. This proves the asserted inclusion of
essential images in $\C/[\Isub]$.
\end{proof}

\subsection{The unit and the cone condition}

\begin{lemma}\label{lem:unit-half}
Let $Z\in\Zsub$. If
$$
  \alpha_Z:Z\longrightarrow\Omega\Sigma Z
$$
is an isomorphism in $\ul{\Zsub}$, then the following stronger statement holds in $\C$:
$$
  Z\langle1\rangle^{-}\in\Tti^{+}.
$$
Consequently, if every component of the unit $\alpha$ is invertible, then
$$
  \ul{\Zsub\langle1\rangle}^{-}\subseteq\ul{\Tti}^{+}.
$$
\end{lemma}

\begin{proof}
Put
$$
  U=Z\langle1\rangle,
  \quad W=\Sigma Z,
  \quad T=W\langle-1\rangle.
$$
By the chosen relative triangles and $(\mathrm{MT3})$, we have
$$
 U\in\Cone_{\E^{\Isub}}(\Zsub,\Zsub)\subseteq\Uti,
 \quad
 T\in\CoCone_{\E_{\Isub}}(\Zsub,\Zsub)\subseteq\Tti.
$$
Choose a representative
$$
  \widetilde\alpha_Z\colon Z\longrightarrow\Omega\Sigma Z
$$
of $\alpha_Z$. Then $\omega T=\Omega W=\Omega\Sigma Z$, and the coreflection triangle \eqref{eq:omega-coreflection} is
\begin{equation}\label{eq:alpha-coreflection}
  V_T\longrightarrow\Omega\Sigma Z\xrightarrow{h_T}T\dashrightarrow,
~~~\mbox{where}~ V_T\in\Vsub.
\end{equation}

Since $\Zsub$ is full, the isomorphism $\alpha_Z$ in $\ul{\Zsub}$ is also an isomorphism between the same objects in $\C/[\Isub]$. By Lemma~\ref{lem:stable-lift}, we can therefore choose $J_0,J_1\in\Isub$ and an isomorphism
\begin{equation*}
  A=
  \begin{pmatrix}
    \widetilde\alpha_Z & a\\ b & c
  \end{pmatrix}:
  Z\oplus J_0\xrightarrow{\sim}\Omega\Sigma Z\oplus J_1.
\end{equation*}
Set
$$
  Y=Z\oplus J_0,
  \quad V'=V_T\oplus J_1.
$$
Since $\Isub\subseteq\Vsub$, we have $V'\in\Vsub$. After adding a split triangle to \eqref{eq:alpha-coreflection} and transporting it along $A^{-1}$, we obtain an $\s_{\Isub}$-triangle
\begin{equation}\label{eq:alpha-first-input}
  V'\longrightarrow Y\xrightarrow{r}T\dashrightarrow,
\end{equation}
where
\begin{equation*}
  r=\begin{pmatrix}h_T&0\end{pmatrix}A.
\end{equation*}
Thus the restriction of $r$ to the $Z$-summand of $Y$ is $h_T\widetilde\alpha_Z$.

After taking the direct sum of the fixed cosyzygy triangle \eqref{eq:cosyzygy} with the split triangle on $J_0$, we obtain an $\s_{\Isub}$-triangle
\begin{equation}\label{eq:alpha-second-input}
  Y\longrightarrow I''\longrightarrow U
  \xdash{\widetilde\lambda^Z},
  \qquad I''=I^Z\oplus J_0\in\Isub,
\end{equation}
where $\widetilde\lambda^Z$ has $\lambda^Z$ as its $Z$-component and zero on $J_0$.

Apply $(\mathrm{ET4})$ in $(\C,\E_{\Isub},\s_{\Isub})$ to \eqref{eq:alpha-first-input} and \eqref{eq:alpha-second-input}. We obtain an object $F_Z$ and $\s_{\Isub}$-triangles
\begin{equation}\label{eq:alpha-output-one}
  V'\longrightarrow I''\longrightarrow F_Z\dashrightarrow
\end{equation}
and
\begin{equation}\label{eq:alpha-output-two}
  T\longrightarrow F_Z\longrightarrow U
  \xdash{\eta_Z}.
\end{equation}
We display the four conflations in the following $(\mathrm{ET4})$ diagram:
$$
\begin{CD}
V' @>>> Y @>{r}>> T\\
@V{1_{V'}}VV @VVV @VVV\\
V' @>>> I'' @>>> F_Z\\
@. @VVV @VVV\\
@. U @= U
\end{CD}
$$
This is the standard $(\mathrm{ET4})$ diagram of Nakaoka--Palu \cite[Definition~2.12]{NakaokaPalu}. The top row is \eqref{eq:alpha-first-input}. The middle column is \eqref{eq:alpha-second-input}, and the remaining row and column are \eqref{eq:alpha-output-one} and \eqref{eq:alpha-output-two}.

The extension compatibility in $(\mathrm{ET4})$ identifies the extension in \eqref{eq:alpha-output-two} with the pushforward of $\widetilde\lambda^Z$ along $r$. Hence
\begin{equation*}
  \eta_Z=r_{*}\widetilde\lambda^Z
  =(h_T\widetilde\alpha_Z)_{*}\lambda^Z
  =h_T\widetilde\alpha_Z\lambda^Z.
\end{equation*}
By Lemma~\ref{lem:roundtrip}(ii),
\begin{equation}\label{eq:eta-alpha-identify}
  \eta_Z=-\lambda_W h^U,
  \quad h^U=h^{Z\langle1\rangle}.
\end{equation}
The extension $\lambda_W h^U$ is exactly the extension defining $U^{-}$ in Iitsuka's construction, because $\sigma U=\Sigma Z=W$, and it belongs to $\E^{\Isub}_{\Isub}(U,T)$. Thus \eqref{eq:eta-alpha-identify} places $\eta_Z$ in the intersection relative structure. Let
$$
  T\longrightarrow U^-\xrightarrow{s_U}U\xdash{\lambda_Wh^U}
$$
be the triangle defining $U^-$. By additivity of $\s$, the extension $-\lambda_Wh^U$ is realized by the same middle term after changing the sign of the first morphism. Therefore \eqref{eq:alpha-output-two} has middle term isomorphic to $U^-$. Hence
\begin{equation}\label{eq:F-is-U-minus}
  F_Z\cong U^{-}=Z\langle1\rangle^{-}.
\end{equation}

Finally, \eqref{eq:alpha-output-one} has first term $V'\in\Vsub$ and middle term $I''\in\Isub$. Thus
$$
  F_Z\in\Cone_{\E_{\Isub}}(\Vsub,\Isub)=\Tti^{+}.
$$
Using \eqref{eq:F-is-U-minus}, we conclude that
$$
  Z\langle1\rangle^{-}\in\Tti^{+}.
$$
If $U'\in\Zsub\langle1\rangle$, then $\ul{U'}$ is isomorphic to the image
of a chosen object $Z\langle1\rangle$. Applying $(\,\cdot\,)^{-}$ shows
$\ul{(U')^{-}}\cong\ul{Z\langle1\rangle^{-}}$. The object on the right is
the image of an object of $\Tti^{+}$. Hence the claimed inclusion of
essential images holds in $\C/[\Isub]$.
\end{proof}

\subsection{Proof of the characterization theorem}

\begin{theorem}\label{thm:main}
Let $(\Ssub,\Zsub,\Vsub)$ be a premutation triple and let
$$
  (\underlineZ,\Sigma,\Omega,\nabla,\Delta),
  \quad \underlineZ=\Zsub/[\Isub],
$$
be the canonical pretriangulated category constructed in \cite[Theorem~3.37]{Iitsuka}. Then the following are equivalent.
\begin{enumerate}
\item $(\underlineZ,\Sigma,\Omega,\nabla,\Delta)$ is triangulated.
\item $(\Ssub,\Zsub,\Vsub)$ satisfies $(\mathrm{MT4})'$.
\item $(\Ssub,\Zsub,\Vsub)$ satisfies $(\mathrm{MT4})$.
\end{enumerate}
\end{theorem}

\begin{proof}
Iitsuka proves $(2)\Leftrightarrow(3)$ in \cite[Remark~4.9(1)]{Iitsuka} and $(2)\Rightarrow(1)$ in \cite[Theorem~4.8]{Iitsuka}. We prove $(1)\Rightarrow(2)$.

Assume that the canonical pretriangulated structure is triangulated. Then its suspension functor $\Sigma$ is an autoequivalence. By Theorem~\ref{thm:iitsuka-pretri}, $\Sigma$ is left adjoint to $\Omega$. Standard adjunction theory shows that a right adjoint of an equivalence is a quasi-inverse, see Mac~Lane \cite[Chapter~IV, Sections~1 and~4]{MacLane}. Hence $(\Sigma,\Omega)$ is an adjoint equivalence. In particular, every component
$$
  \alpha_Z\colon Z\longrightarrow\Omega\Sigma Z,
  \quad
  \beta_Z\colon \Sigma\Omega Z\longrightarrow Z
$$
is an isomorphism in $\underlineZ$.

Lemmas~\ref{lem:counit-half} and~\ref{lem:unit-half} directly give the
following inclusions in $\C/[\Isub]$.
$$
 \ul{\Zsub\langle1\rangle}^{-}\subseteq\ul{\Tti}^{+},
 \quad
 \ul{\Zsub\langle-1\rangle}^{+}\subseteq\ul{\Uti}^{-}.
$$
We have obtained the two clauses of $(\mathrm{MT4})'$, and hence $(1)\Rightarrow(2)$.
\end{proof}

\begin{remark}\label{rem:one-sided}
We can also apply the two lemmas separately, without assuming that the adjunction is an equivalence. Namely,
$$
  \beta_Z\text{ invertible for all }Z
  \quad\Longrightarrow\quad
  \ul{\Zsub\langle-1\rangle}^{+}\subseteq\ul{\Uti}^{-},
$$
whereas
$$
  \alpha_Z\text{ invertible for all }Z
  \quad\Longrightarrow\quad
  \ul{\Zsub\langle1\rangle}^{-}\subseteq\ul{\Tti}^{+}.
$$
Thus the two clauses of $(\mathrm{MT4})'$ are necessary for the unit and the
counit of the mutation adjunction, respectively, to be invertible.
\end{remark}

\section{Consequences and special cases}\label{sec:applications}

In this section, we use the characterization to give a functorial reformulation of $(\mathrm{MT4})$, to sharpen the comparison with concentric twin cotorsion pairs, and to isolate the remaining obstruction for reducible triples.

\subsection{Functorial criteria}

\begin{corollary}\label{cor:functorial-characterization}
Let $(\Ssub,\Zsub,\Vsub)$ be a premutation triple. The following conditions are equivalent.
\begin{enumerate}
\item $(\Ssub,\Zsub,\Vsub)$ satisfies $(\mathrm{MT4})$.
\item The canonical pretriangulated structure on $\Zsub/[\Isub]$ is triangulated.
\item The mutation functor $\Sigma$ is an autoequivalence.
\item The mutation functor $\Omega$ is an autoequivalence.
\item The adjoint pair $(\Sigma,\Omega)$ is an adjoint equivalence.
\item Both the unit $\alpha$ and the counit $\beta$ are natural isomorphisms.
\end{enumerate}
\end{corollary}

\begin{proof}
By Theorem~\ref{thm:main}, we have $(1)\Leftrightarrow(2)$. From the definition of a triangulated category, we obtain $(2)\Rightarrow(3)$ and $(2)\Rightarrow(4)$. If (3) holds, the right adjoint $\Omega$ of the equivalence $\Sigma$ is a quasi-inverse, so (5) holds; the implication $(4)\Rightarrow(5)$ is dual. Standard properties of adjunctions give $(5)\Leftrightarrow(6)$, see Mac~Lane \cite[Chapter~IV, Sections~1 and~4]{MacLane}. Finally, under (5), the functors $\Sigma$ and $\Omega$ are mutually quasi-inverse, so Iitsuka \cite[Corollary~3.38]{Iitsuka} gives (2).
\end{proof}

We can also formulate the two converse lemmas in terms of full faithfulness.

\begin{corollary}\label{cor:full-faithfulness}
Let $(\Ssub,\Zsub,\Vsub)$ be a premutation triple.
\begin{enumerate}
\item If $\Sigma$ is fully faithful, then the following inclusion holds in $\C/[\Isub]$.
$$
  \ul{\Zsub\langle1\rangle}^{-}\subseteq\ul{\Tti}^{+}.
$$

\item If $\Omega$ is fully faithful, then the following inclusion holds in $\C/[\Isub]$.
$$
  \ul{\Zsub\langle-1\rangle}^{+}\subseteq\ul{\Uti}^{-}.
$$
\item If both $\Sigma$ and $\Omega$ are fully faithful, then $(\mathrm{MT4})$ holds.
\end{enumerate}
\end{corollary}

\begin{proof}
For the adjunction between $\Sigma$ and $\Omega$, standard adjunction theory identifies full faithfulness of $\Sigma$ with invertibility of the unit,  see Mac~Lane \cite[Chapter~IV, Section~3]{MacLane}. We obtain (1) from Lemma~\ref{lem:unit-half}. Dually, full faithfulness of $\Omega$ is equivalent to invertibility of the counit, so Lemma~\ref{lem:counit-half} gives (2). If both functors are fully faithful, both the unit and the counit are invertible. Hence the adjunction is an equivalence, and Corollary~\ref{cor:functorial-characterization} gives (3).
\end{proof}

\subsection{Concentric twin cotorsion pairs}

We now specialize to concentric twin cotorsion pairs. As Iitsuka specifies in the opening paragraph of \cite[Appendix~A]{Iitsuka}, the subcategories in this subsection are understood to be closed under direct summands.

\begin{definition}\label{def:concentric-tcp}
Following Liu and Nakaoka \cite[Definitions~2.1 and~2.3]{LiuNakaoka}, a pair $(\mathcal U,\mathcal V)$ is a cotorsion pair if
$$
  \E(\mathcal U,\mathcal V)=0,\quad
  \C=\Cone(\mathcal V,\mathcal U)
  =\CoCone(\mathcal V,\mathcal U).
$$
A pair of cotorsion pairs $((\Ssub,\mathcal T),(\mathcal U,\Vsub))$ is a \emph{twin cotorsion pair} if $\Ssub\subseteq\mathcal U$, and it is \emph{concentric} if
$$
  \Ssub\cap\mathcal T=\mathcal U\cap\Vsub.
$$
Iitsuka recalls the same definitions in \cite[Definitions~A.2 and~A.4]{Iitsuka}, see also Nakaoka \cite[Definition~3.3]{Nakaoka} and Zheng--Wei \cite[Section~2]{ZhengWei}.
\end{definition}

Let $\C$ be triangulated, let $((\Ssub,\mathcal T),(\mathcal U,\Vsub))$ be a concentric twin cotorsion pair, and put
$$
  \Zsub=\mathcal T\cap\mathcal U.
$$
By \cite[Example~3.13(5)]{Iitsuka}, the associated triple $(\Ssub,\Zsub,\Vsub)$ is a premutation triple. Its ideal subcategory is the common intersection in the concentric condition. Since $\Ssub\subseteq\mathcal U$ and $\Vsub\subseteq\mathcal T$,
$$
  \Ssub\cap\Zsub=\Ssub\cap\mathcal T
  =\mathcal U\cap\Vsub=\Zsub\cap\Vsub=\Isub.
$$
Iitsuka proves in \cite[Example~3.15(4)]{Iitsuka} that the mutation functors agree with those for concentric twin cotorsion pairs. For completeness, Iitsuka \cite[Definition~A.19 and Lemma~A.20]{Iitsuka} also proves that an $\s$-triangle $A\to B\to C\dashrightarrow$ with $A,B\in\Zsub$ has $C\in\mathcal U$ exactly when its first morphism is $\Isub$-monic; dually, an $\s$-triangle $A\to B\to C\dashrightarrow$ with $B,C\in\Zsub$ has $A\in\mathcal T$ exactly when its second morphism is $\Isub$-epic. Iitsuka uses these triangles in \cite[Definition~A.21]{Iitsuka}. Thus the pretriangulated structure in \cite[Theorem~A.22]{Iitsuka} is the specialization of the canonical structure in \cite[Theorem~3.37]{Iitsuka} to concentric twin cotorsion pairs.

\begin{definition}\label{def:tcp-inclusions}
Nakaoka's construction uses the following two subcategories:
$$
  \mathcal N^{\mathrm i}=\Cone(\Vsub,\Ssub),
  \quad
  \mathcal N^{\mathrm f}=\CoCone(\Vsub,\Ssub).
$$
Thus, $\mathcal N^{\mathrm i}$ consists of the objects which occur as third
terms of triangles whose first term lies in $\Vsub$ and whose middle term lies
in $\Ssub$; $\mathcal N^{\mathrm f}$ consists of the objects which occur as
first terms of triangles whose middle term lies in $\Vsub$ and whose third term
lies in $\Ssub$. We will use the following two quotient inclusions from Iitsuka
\cite[Condition~A.27]{Iitsuka}:
$$
  \ul{\Zsub\langle1\rangle}^{-}\subseteq\ul{\mathcal N^{\mathrm i}},
  \quad
  \ul{\Zsub\langle-1\rangle}^{+}\subseteq\ul{\mathcal N^{\mathrm f}}.
$$
Iitsuka records these inclusions immediately after observing that Nakaoka's
proof works under precisely these two weaker hypotheses, see \cite[the
paragraph preceding Condition~A.27]{Iitsuka}.
\end{definition}

Iitsuka observes that this is the condition denoted by $(\mathrm{MT4}^{-})$ for the associated premutation triple \cite[Remark~4.9(3)]{Iitsuka}. In general, $(\mathrm{MT4})$ implies $(\mathrm{MT4}^{-})$, and the converse is not part of the general premutation theory. For concentric twin cotorsion pairs the converse follows from the main theorem.

We first record the triangulatedness statement under these two inclusions. The proof below makes explicit why Condition~A.27, rather than the stronger assumptions in \cite[Theorem~A.26]{Iitsuka}, is sufficient.

\begin{lemma}\label{lem:tcp-inclusions-triangulated}
Let $((\Ssub,\mathcal T),(\mathcal U,\Vsub))$ be a concentric twin cotorsion pair in a triangulated category. If the two inclusions in Definition~\ref{def:tcp-inclusions} hold, then the canonical pretriangulated category
$$
 (\ul{\Zsub},\Sigma,\Omega,\nabla,\Delta)
$$
is triangulated.
\end{lemma}

\begin{proof}
Iitsuka observes in the paragraph preceding
\cite[Condition~A.27]{Iitsuka} that the proof of
\cite[Theorem~A.26(1)]{Iitsuka} remains valid under the two weaker inclusions
in Definition~\ref{def:tcp-inclusions}. These inclusions are precisely
Conditions $(\mathrm{Nak1})$ and $(\mathrm{Nak2})$ of
\cite[Condition~A.27]{Iitsuka}. The argument is based on
\cite[Remark~5.8]{Nakaoka}, which gives natural isomorphisms
$$
  \Sigma\Omega\cong\Id_{\ul{\Zsub}},
  \qquad
  \Omega\Sigma\cong\Id_{\ul{\Zsub}}.
$$
Hence $\Sigma$ and $\Omega$ are mutually quasi-inverse. By
\cite[Example~3.15(4) and Theorem~A.22]{Iitsuka}, the mutation functors and
the right and left triangles in this construction agree with the canonical
structure associated with the premutation triple
$(\Ssub,\Zsub,\Vsub)$. Therefore the canonical pretriangulated category is
triangulated.
\end{proof}

\begin{theorem}\label{thm:tcp-application}
Let $((\Ssub,\mathcal T),(\mathcal U,\Vsub))$ be a concentric twin cotorsion pair in a triangulated category and let $(\Ssub,\Zsub,\Vsub)$ be its associated premutation triple. The following conditions are equivalent.
\begin{enumerate}
\item The following two inclusions hold:
$$
  \ul{\Zsub\langle1\rangle}^{-}\subseteq\ul{\mathcal N^{\mathrm i}},
  \quad
  \ul{\Zsub\langle-1\rangle}^{+}\subseteq\ul{\mathcal N^{\mathrm f}}.
$$
\item The associated premutation triple satisfies $(\mathrm{MT4}^{-})$.
\item The associated premutation triple satisfies $(\mathrm{MT4})$.
\item The canonical pretriangulated structure on $\Zsub/[\Isub]$ is triangulated.
\end{enumerate}
\end{theorem}

\begin{proof}
Iitsuka proves $(1)\Leftrightarrow(2)$ in \cite[Remark~4.9(3) and Condition~A.27]{Iitsuka} and records $(3)\Rightarrow(2)$ in \cite[Remark~4.9(3)]{Iitsuka}.

By Lemma~\ref{lem:tcp-inclusions-triangulated}, $(1)$ implies (4). By Theorem~\ref{thm:main}, (4) implies (3). Thus all the conditions are equivalent.
\end{proof}

Thus, for premutation triples arising from concentric twin cotorsion pairs, we remove the apparent gap between $(\mathrm{MT4})$ and $(\mathrm{MT4}^{-})$. We also recall the two special classes used below. For subcategories $\mathcal X,\mathcal Y$ of a triangulated category, we write $\mathcal X*\mathcal Y$ for the full subcategory of objects $M$ which occur in a distinguished triangle
$$
  X\longrightarrow M\longrightarrow Y\longrightarrow X[1],
  \quad\mbox{where}~ X\in\mathcal X,~ Y\in\mathcal Y.
$$
This is the extension product used by Nakaoka \cite[Definition~1.1 and the paragraph following it]{Nakaoka}. In Iitsuka's notation \cite[Definition~A.23]{Iitsuka}, a concentric twin cotorsion pair is \emph{Hovey} if $\mathcal N^{\mathrm i}=\mathcal N^{\mathrm f}$, and it is \emph{heart-equivalent} if
$$
  \add(\Ssub*\mathcal T)=\add(\mathcal U*\Vsub).
$$
One consequence is the following statement.

\begin{corollary}\label{cor:heart-equivalent}
Let $((\Ssub,\mathcal T),(\mathcal U,\Vsub))$ be a concentric twin cotorsion pair in a triangulated category. If it is Hovey or heart-equivalent, then the associated premutation triple $(\Ssub,\Zsub,\Vsub)$ satisfies $(\mathrm{MT4})$. Hence it is a mutation triple.
\end{corollary}

\begin{proof}
Iitsuka \cite[Theorem~A.26]{Iitsuka}, using results of Nakaoka \cite[Corollary~5.11, Corollary~6.12, and Proposition~6.2]{Nakaoka}, proves that the canonical pretriangulated category is triangulated for Hovey and heart-equivalent concentric twin cotorsion pairs. We conclude from Theorem~\ref{thm:main} that the associated premutation triple satisfies $(\mathrm{MT4})$.
\end{proof}

\begin{remark}\label{rem:heart-new}
Iitsuka proves the Hovey case of Corollary~\ref{cor:heart-equivalent} directly in \cite[Proposition~4.13]{Iitsuka}. In the heart-equivalent case, Iitsuka \cite[Remark~4.12]{Iitsuka} proves that the different triple $(\Isub,\Zsub,\Isub)$ is a mutation triple. Our characterization additionally yields $(\mathrm{MT4})$ for the associated triple $(\Ssub,\Zsub,\Vsub)$ itself.
\end{remark}

We finally recall Nakaoka's Condition~5.2 because it is used in the next corollary. The adjunctions are taken in the ideal quotient categories. More precisely, $\omega_{\mathcal U}:\C/[\Isub]\to\mathcal U/[\Isub]$ is right adjoint to the inclusion $\mathcal U/[\Isub]\hookrightarrow\C/[\Isub]$, and $\sigma_{\mathcal T}:\C/[\Isub]\to\mathcal T/[\Isub]$ is left adjoint to the inclusion $\mathcal T/[\Isub]\hookrightarrow\C/[\Isub]$. Their restrictions give functors $\omega:\mathcal T/[\Isub]\to\Zsub/[\Isub]$ and $\sigma:\mathcal U/[\Isub]\to\Zsub/[\Isub]$ as recalled in \cite[Lemma~A.8(2)--(3)]{Iitsuka}. Nakaoka constructs in \cite[Proposition~3.13]{Nakaoka} a canonical comparison morphism
$$
  \mu_X:\sigma(\omega_{\mathcal U}(X))
  \longrightarrow
  \omega(\sigma_{\mathcal T}(X))
  \quad\text{in }\Zsub/[\Isub].
$$
His Condition~5.2 \cite[Condition~5.2]{Nakaoka} consists of the following two requirements. For every $X\in\mathcal T*\mathcal U$, the morphism $\mu_X$ is an isomorphism in $\Zsub/[\Isub]$, and
$$
  \mathcal U\cap\mathcal N^{\mathrm f}=\Ssub,
  \quad
  \mathcal T\cap\mathcal N^{\mathrm i}=\Vsub.
$$
Nakaoka calls these requirements Conditions $(\mathrm I)$ and $(\mathrm{II})$, respectively.

\begin{corollary}\label{cor:nakaoka-I-II}
Let $((\Ssub,\mathcal T),(\mathcal U,\Vsub))$ be a concentric twin cotorsion pair in a triangulated category. If it satisfies Nakaoka's Conditions $(\mathrm I)$ and $(\mathrm{II})$ recalled above, then the associated premutation triple $(\Ssub,\Zsub,\Vsub)$ is a mutation triple.
\end{corollary}

\begin{proof}
The implication diagram in Iitsuka \cite[Remark~5.14]{Iitsuka} records that
Nakaoka's Conditions $(\mathrm I)$ and $(\mathrm{II})$ imply the condition
denoted there by $(\mathrm{Nak})$. By \cite[Condition~A.27]{Iitsuka}, this is
exactly the pair of quotient inclusions in
Definition~\ref{def:tcp-inclusions}. The assertion now follows from
Theorem~\ref{thm:tcp-application}.
\end{proof}

\subsection{Reducible triples}

Iitsuka \cite[Condition~5.1, Condition~5.6, and Definition~5.8]{Iitsuka} introduced \emph{prereducible triples} by replacing $(\mathrm{MT3})$ with $(\mathrm{RT3})$, and \emph{reducible triples} by additionally imposing $(\mathrm{RT4})$. Iitsuka notes in \cite[Remark~5.2]{Iitsuka} that $(\mathrm{MT3})$ and $(\mathrm{RT3})$ need not coincide. By \cite[Corollary~5.4]{Iitsuka}, one defines the functors $\Sigma,\Omega$ and the classes of right and left triangles by the same formulas as for premutation triples, and they give a canonical pretriangulated category. If $(\mathrm{RT4})$ also holds, Iitsuka \cite[Corollary~5.9]{Iitsuka} proves that this same quintuple is triangulated.

\begin{definition}\label{def:reducible}
A triple satisfying $(\mathrm{MT1})$ and $(\mathrm{MT2})$ satisfies $(\mathrm{RT3})$ if
$$
  \Cone_{\E^{\Isub}}(\Zsub,\Zsub)
  \subseteq
  \CoCone_{\E^{\Isub}_{\Isub}}(\Zsub,\Ssub),
  \quad
  \CoCone_{\E_{\Isub}}(\Zsub,\Zsub)
  \subseteq
  \Cone_{\E^{\Isub}_{\Isub}}(\Vsub,\Zsub),
$$
and if $\Ssub$, $\Zsub$ and $\Vsub$ are extension closed in $(\C,\E^{\Isub}_{\Isub},\s^{\Isub}_{\Isub})$. It satisfies $(\mathrm{RT4})$ if $\Isub$ is strongly contravariantly finite in $\Ssub$, strongly covariantly finite in $\Vsub$, and
$$
  \Ssub\langle-1\rangle=\Vsub\langle1\rangle.
$$
A \emph{prereducible triple} is a triple satisfying $(\mathrm{MT1})$, $(\mathrm{MT2})$ and $(\mathrm{RT3})$. A \emph{reducible triple} is a prereducible triple that also satisfies $(\mathrm{RT4})$.
\end{definition}

\begin{theorem}\label{thm:reducible-application}
Let $(\Ssub,\Zsub,\Vsub)$ be a reducible triple. The following conditions are equivalent.
\begin{enumerate}
\item $(\Ssub,\Zsub,\Vsub)$ is a premutation triple.
\item $(\Ssub,\Zsub,\Vsub)$ is a mutation triple.
\item The subcategories $\Ssub$ and $\Zsub$ are closed under extensions in
$$
  (\C,\E^{\Isub},\s^{\Isub}),
$$
and the subcategories $\Zsub$ and $\Vsub$ are closed under extensions in
$$
  (\C,\E_{\Isub},\s_{\Isub}).
$$
\end{enumerate}
In particular, for a reducible triple, the only possible obstruction to being a mutation triple is the extension closure requirement in $(\mathrm{MT3})$.
\end{theorem}

\begin{proof}
We first prove $(1)\Leftrightarrow(2)$. Iitsuka \cite[Corollary~5.4]{Iitsuka} defines the canonical pretriangulated quintuple for a prereducible triple by the same $\Sigma,\Omega,\nabla,\Delta$ as in the premutation construction, and Iitsuka \cite[Corollary~5.9]{Iitsuka} proves that this quintuple is triangulated for a reducible triple. Therefore, if a reducible triple is also a premutation triple, we may apply Theorem~\ref{thm:main} to the same canonical structure and obtain $(\mathrm{MT4})$. Hence the triple is a mutation triple. The converse follows from the definition.

Here the two canonical structures do agree. Under $(\mathrm{MT2})$, the
approximation triangles which define $\langle1\rangle$ and
$\langle-1\rangle$ may be chosen in the intersection relative structure.
Under $(\mathrm{RT3})$, the reflection and coreflection triangles needed for
the standard right and left triangles may also be chosen there. The
independence of these choices in \cite[Corollary~5.4]{Iitsuka} identifies the
resulting functors and triangle classes with those of the premutation
construction.

It remains to compare (1) and (3). A reducible triple already satisfies $(\mathrm{MT1})$, $(\mathrm{MT2})$ and $(\mathrm{RT3})$. The first clause of $(\mathrm{RT3})$ gives
$$
 \Cone_{\E^{\Isub}}(\Zsub,\Zsub)
 \subseteq
 \CoCone_{\E^{\Isub}_{\Isub}}(\Zsub,\Ssub)
 \subseteq
 \CoCone_{\E^{\Isub}}(\Zsub,\Ssub).
$$
Thus the right cone--cocone clause of $(\mathrm{MT3})$ is automatic. The second clause of $(\mathrm{RT3})$ gives dually
$$
 \CoCone_{\E_{\Isub}}(\Zsub,\Zsub)
 \subseteq
 \Cone_{\E^{\Isub}_{\Isub}}(\Vsub,\Zsub)
 \subseteq
 \Cone_{\E_{\Isub}}(\Vsub,\Zsub).
$$
Hence the left cone--cocone clause of $(\mathrm{MT3})$ is automatic as well. Consequently, under the defining assumptions of a reducible triple, condition $(\mathrm{MT3})$ is equivalent to its extension closure clause. This clause is exactly (3). Therefore (1) and (3) are equivalent.
\end{proof}

\begin{proposition}\label{prop:reducible-counterexample}
There exists a reducible triple which is not a premutation triple. Consequently, a reducible triple need not be a mutation triple.
\end{proposition}

\begin{proof}
Let $Q$ be the linearly oriented quiver
$$
  1\longrightarrow2\longrightarrow3\longrightarrow4
$$
and let $\C=\mod kQ$, with its standard exact structure. For $1\leq a\leq b\leq4$, write $M_{ab}$ for the interval module supported on the vertices $a,a+1,\ldots,b$, with identity maps on the arrows inside its support.

We shall use two elementary calculations for interval modules. They are included to make the verification of the relative conditions transparent. One has
$$
 \dim_k\Hom(M_{ab},M_{cd})=
 \begin{cases}
  1 & \text{if }c\leq a\leq d\leq b,\\
  0 & \text{otherwise},
 \end{cases}
$$
and
$$
 \dim_k\operatorname{Ext}^1(M_{ab},M_{cd})=
 \begin{cases}
  1 & \text{if }a<c\leq b+1\leq d,\\
  0 & \text{otherwise}.
 \end{cases}
$$
The first formula follows directly from the commutativity relations for a morphism of interval representations. For the second, if $b<4$, use the projective resolution
$$
  0\longrightarrow M_{b+1,4}\longrightarrow M_{a4}
  \longrightarrow M_{ab}\longrightarrow0.
$$
If $b=4$, the module $M_{a4}$ is projective.

For later reference, the groups $\operatorname{Ext}^1(S,I)$ with $S$ an indecomposable object of $\Ssub$ and $I$ an indecomposable object of $\Isub$ are displayed in the following table. Each entry is its $k$-dimension.
$$
\begin{array}{c|ccc}
 \dim_k\operatorname{Ext}^1(S,I)
   & I=M_{11} & I=M_{23} & I=M_{34}\\ \hline
 S=M_{11} &0&1&0\\
 S=M_{22} &0&0&1\\
 S=M_{23} &0&0&1\\
 S=M_{34} &0&0&0
\end{array}
$$

Set
$$
 \Isub=\add(M_{11}\oplus M_{23}\oplus M_{34}),
 \quad
 \Zsub=\Isub,
$$
$$
 \Ssub=\add(\Isub\oplus M_{22}),
 \quad
 \Vsub=\add(\Isub\oplus M_{44}).
$$
We verify that $(\Ssub,\Zsub,\Vsub)$ is reducible.

Condition $(\mathrm{MT1})$ is immediate. Indeed,
$$
  \Ssub\cap\Zsub=\Zsub\cap\Vsub=\Isub,
$$
and $\Isub$ is strongly functorially finite in $\Zsub=\Isub$.

We next check $(\mathrm{MT2})$. Since $\Zsub=\Isub$, the identity morphism of an object of $\Isub$ shows that
$$
  \E^{\Isub}(\Ssub,\Zsub)=0,
  \quad
  \E_{\Isub}(\Zsub,\Vsub)=0.
$$
The same observation gives
$$
 \CoCone_{\E_{\Isub}}(\Isub,\Zsub)=\Isub,
 \quad
 \Cone_{\E^{\Isub}}(\Zsub,\Isub)=\Isub.
$$
It remains to verify
$$
  \E_{\Isub}(\Ssub,\Isub)=0,
  \quad
  \E^{\Isub}(\Isub,\Vsub)=0.
$$
By additivity, it is enough to check indecomposable pairs. For the first equality, the only pairs of indecomposable objects $S\in\Ssub$ and $I\in\Isub$ with $\operatorname{Ext}^1(S,I)\neq0$ give, up to isomorphism, the following non-split exact sequences
$$
  0\longrightarrow M_{23}\longrightarrow M_{13}
  \longrightarrow M_{11}\longrightarrow0,
$$
$$
  0\longrightarrow M_{34}\longrightarrow M_{24}
  \longrightarrow M_{22}\longrightarrow0,
$$
and
$$
  0\longrightarrow M_{34}\longrightarrow M_{24}\oplus M_{33}
  \longrightarrow M_{23}\longrightarrow0.
$$
None of the three deflations is $\Isub$-epic. In the first case the identity of $M_{11}$ does not lift because $\Hom(M_{11},M_{13})=0$. In the second case the nonzero map $M_{23}\to M_{22}$ does not lift because $\Hom(M_{23},M_{24})=0$. In the third case the identity of $M_{23}$ does not lift because
$$
  \Hom(M_{23},M_{24})=\Hom(M_{23},M_{33})=0.
$$
Each nonzero group in the table is one-dimensional over $k$. Moreover,
$\E_{\Isub}(S,I)$ is a $k$-linear subspace of
$\operatorname{Ext}^1(S,I)$. In every relevant case, the displayed non-split
exact sequence represents a nonzero generator. The preceding calculation
shows that this generator does not belong to $\E_{\Isub}(S,I)$. Hence
$\E_{\Isub}(\Ssub,\Isub)=0$.

For the second equality, extensions whose second argument belongs to $\Isub$ cannot be nonzero elements of $\E^{\Isub}$. Hence only the extra object $M_{44}\in\Vsub$ needs to be considered. The only relevant non-split exact sequence is
$$
  0\longrightarrow M_{44}\longrightarrow M_{24}
  \longrightarrow M_{23}\longrightarrow0.
$$
The nonzero map $M_{44}\to M_{34}$ does not extend across $M_{24}$ because
$$
  \Hom(M_{24},M_{34})=0.
$$
Therefore the inflation is not $\Isub$-monic, and $\E^{\Isub}(\Isub,\Vsub)=0$. This proves $(\mathrm{MT2})$.

We turn to $(\mathrm{RT3})$. Since $\Zsub=\Isub$, every $\E^{\Isub}$-conflation whose first two terms belong to $\Zsub$ splits. Indeed, its first term belongs to $\Isub$, so the defining $\Isub$-monic condition forces the identity of that term to factor through the inflation. Consequently,
$$
  \Cone_{\E^{\Isub}}(\Zsub,\Zsub)=\Isub.
$$
Dually,
$$
  \CoCone_{\E_{\Isub}}(\Zsub,\Zsub)=\Isub.
$$
The required two inclusions in $(\mathrm{RT3})$ now follow from split conflations.

It remains to check extension closure in the intersection relative structure. Let $S_1,S_2$ be indecomposable objects of $\Ssub$. If $S_1\in\Isub$, then $\E^{\Isub}(S_2,S_1)=0$, and if $S_2\in\Isub$, then $\E_{\Isub}(S_2,S_1)=0$. The only remaining pair is $(S_2,S_1)=(M_{22},M_{22})$, and
$$
 \operatorname{Ext}^1(M_{22},M_{22})=0.
$$
Hence $\E^{\Isub}_{\Isub}(\Ssub,\Ssub)=0$. By biadditivity, this equality holds for arbitrary finite direct sums and direct summands of the displayed indecomposables. Therefore every intersection-relative conflation with end terms in $\Ssub$ splits, so $\Ssub$ is extension closed. The same argument, using $\operatorname{Ext}^1(M_{44},M_{44})=0$, gives $\E^{\Isub}_{\Isub}(\Vsub,\Vsub)=0$. It is immediate that $\E^{\Isub}_{\Isub}(\Zsub,\Zsub)=0$. Thus the extension closure clause of $(\mathrm{RT3})$ holds.

Finally, consider $(\mathrm{RT4})$. The canonical epimorphism
$$
  M_{23}\longrightarrow M_{22}
$$
is a right $\Isub$-approximation. Its kernel is $M_{33}$. Indeed,
$$
  \Hom(M_{11},M_{22})=\Hom(M_{34},M_{22})=0
$$
and $\dim_k\Hom(M_{23},M_{22})=1$. Thus $\Isub$ is strongly contravariantly finite in $\Ssub$. Dually, the canonical monomorphism
$$
  M_{44}\longrightarrow M_{34}
$$
is a left $\Isub$-approximation with cokernel $M_{33}$, and $\Isub$ is strongly covariantly finite in $\Vsub$.

These approximation properties also determine the two relative shift subcategories. Let $f:I_0\to M_{22}$ be an $\Isub$-epic deflation with $I_0\in\Isub$. The relevant morphism spaces are
$$
 \Hom(M_{11},M_{22})=\Hom(M_{34},M_{22})=0,
 \qquad
 \dim_k\Hom(M_{23},M_{22})=1.
$$
Thus only a summand isomorphic to $M_{23}$ can map nontrivially to $M_{22}$. Since $f$ is $\Isub$-epic, at least one such component is nonzero. After an automorphism of the $M_{23}$-isotypic summand of $I_0$, the morphism $f$ is the canonical epimorphism $M_{23}\to M_{22}$ on one summand and zero on a complementary object of $\Isub$. Therefore
$$
 \ker f\cong M_{33}\oplus I'
 \quad\text{for some }I'\in\Isub.
$$
We briefly justify the passage to arbitrary objects of $\Ssub$. If the target of an $\Isub$-epic deflation has a summand in $\Isub$, the identity of that summand lifts and the corresponding part splits off. It remains to consider a target $M_{22}^{\oplus n}$. After discarding source summands which map trivially, the deflation is represented by a matrix
$$
 M_{23}^{\oplus m}\longrightarrow M_{22}^{\oplus n}.
$$
The $\Isub$-epic property says that this matrix has rank $n$. Row and column operations reduce it to the direct sum of $n$ copies of the canonical epimorphism $M_{23}\to M_{22}$ and a zero map on a complementary object of $\Isub$. Its kernel therefore belongs to $\add(\Isub\oplus M_{33})$. Conversely, $M_{33}$ is obtained from the displayed approximation and each object of $\Isub$ is obtained from a split conflation. Hence
$$
  \Ssub\langle-1\rangle=\add(\Isub\oplus M_{33}).
$$
The dual matrix argument applies to arbitrary objects of $\Vsub$. Here
$$
 \dim_k\Hom(M_{44},M_{34})=1.
$$
It follows that
$$
  \Vsub\langle1\rangle=\add(\Isub\oplus M_{33}).
$$
Thus $(\mathrm{RT4})$ holds, and $(\Ssub,\Zsub,\Vsub)$ is a reducible triple.

It remains to show that it is not premutation. Consider the non-split exact sequence
$$
  0\longrightarrow M_{22}\longrightarrow M_{12}
  \longrightarrow M_{11}\longrightarrow0.
$$
Since
$$
  \Hom(M_{22},\Isub)=0,
$$
its extension class belongs to $\E^{\Isub}(M_{11},M_{22})$. Both end terms lie in $\Ssub$, whereas $M_{12}\notin\Ssub$. Therefore $\Ssub$ is not extension closed in $(\C,\E^{\Isub},\s^{\Isub})$. Condition $(\mathrm{MT3})$ fails, so the reducible triple is not a premutation triple and hence is not a mutation triple.
\end{proof}

\begin{remark}\label{rem:reducible-question}
Iitsuka \cite[the paragraph preceding Corollary~5.9]{Iitsuka} asks whether every reducible triple is a mutation triple. In Proposition~\ref{prop:reducible-counterexample}, we give a negative answer. In our example, the cone and cocone conditions in $(\mathrm{MT3})$ hold. The intersection relative extension closure required by $(\mathrm{RT3})$ also holds. However, extension closure in $\E^{\Isub}$ fails. Thus the extension closure condition in Theorem~\ref{thm:reducible-application} is a genuine obstruction.

In Proposition~\ref{prop:reducible-counterexample}, we have $\Zsub=\Isub$, so the quotient $\ul{\Zsub}$ is zero. Iitsuka's question \cite[the paragraph preceding Corollary~5.9]{Iitsuka} does not exclude this case. The next example shows that the same failure can occur when the quotient is nonzero.

\end{remark}

\begin{example}\label{cor:nonzero-reducible-counterexample}
There exists a reducible triple which is not a premutation triple and whose associated quotient is nonzero.
\end{example}

\begin{proof}
Let $\Gamma$ be the disjoint union of the linearly oriented quiver
$1\to2\to3\to4$ and the one-loop quiver with vertex $5$ and loop
$\varepsilon$. Put $A=k\Gamma/(\varepsilon^2)$. We work in $\C=\mod A$ with
its standard exact structure. On the first component, we retain the interval modules $M_{ab}$ from
Proposition~\ref{prop:reducible-counterexample}. On the second component, let
$L$ be the simple module and let $P$ be the indecomposable projective module.
The module $P$ has length two and occurs in the non-split
exact sequence
$$
  0\longrightarrow L\longrightarrow P\longrightarrow L\longrightarrow0.
$$
The algebra supported at vertex $5$ is $R=k[\varepsilon]/(\varepsilon^2)$. If
$t:R\to k$ sends $a+b\varepsilon$ to $b$, then
$$
 R\longrightarrow \Hom_k(R,k),
 \quad x\longmapsto \bigl(y\longmapsto t(yx)\bigr),
$$
is an isomorphism of $R$-bimodules. Thus $R$ is symmetric, and $P$ is
injective as well as projective.

Every module supported at vertex $5$ belongs to $\add(P\oplus L)$.

We now set
$$
 \Isub=\add(M_{11}\oplus M_{23}\oplus M_{34}\oplus P),
$$
$$
 \Zsub=\add(\Isub\oplus L),\quad
 \Ssub=\add(\Isub\oplus M_{22}),\quad
 \Vsub=\add(\Isub\oplus M_{44}).
$$
We verify that this is the required example. There are no morphisms or
extensions between modules supported on different connected components of
$\Gamma$. Consequently, all the following calculations separate according to
the support of the modules.

On the first component, we use the calculations in
Proposition~\ref{prop:reducible-counterexample} without change. On the
second component, the deflation $P\to L$ is a right $\Isub$-approximation, and
the inflation $L\to P$ is a left $\Isub$-approximation. The kernel of the first
map and the cokernel of the second are both $L$. Hence $\Isub$ is strongly
functorially finite in $\Zsub$, and
$$
  \Zsub\langle-1\rangle=\Zsub=\Zsub\langle1\rangle.
$$
Since $P$ is projective-injective, the four relative extension groups in
$(\mathrm{MT2})$ vanish on the second component. Combining this with the
verification in Proposition~\ref{prop:reducible-counterexample} proves
$(\mathrm{MT1})$ and $(\mathrm{MT2})$.

We next check $(\mathrm{RT3})$. The modules supported at vertex $5$ form the
extension closed subcategory $\add(P\oplus L)$. Moreover, projectivity and
injectivity of $P$ imply that every exact sequence on this component belongs to
both relative structures. Split exact sequences with an end term $P$ therefore
give
$$
 \Cone_{\E^{\Isub}}(\add(P\oplus L),\add(P\oplus L))
 =\add(P\oplus L)
 =\CoCone_{\E^{\Isub}_{\Isub}}(\add(P\oplus L),\add P),
$$
and, dually,
$$
 \CoCone_{\E_{\Isub}}(\add(P\oplus L),\add(P\oplus L))
 =\add(P\oplus L)
 =\Cone_{\E^{\Isub}_{\Isub}}(\add P,\add(P\oplus L)).
$$
Together with the first-component calculation in
Proposition~\ref{prop:reducible-counterexample}, these equalities prove the two
cone--cocone inclusions and the extension closure clause of $(\mathrm{RT3})$.

We still use the approximation maps $M_{23}\to M_{22}$ and $M_{44}\to M_{34}$
from Proposition~\ref{prop:reducible-counterexample}; they remain $\Isub$-approximations,
and their kernel and cokernel are both $M_{33}$. All the other indecomposable
objects of $\Ssub$ and $\Vsub$ belong to $\Isub$. Hence
$$
  \Ssub\langle-1\rangle
  =\add(\Isub\oplus M_{33})
  =\Vsub\langle1\rangle.
$$
This proves $(\mathrm{RT4})$, so the triple is reducible.

The non-split exact sequence
$$
  0\longrightarrow M_{22}\longrightarrow M_{12}
  \longrightarrow M_{11}\longrightarrow0
$$
still belongs to $\E^{\Isub}$. Indeed,
$\Hom(M_{22},M_{11}\oplus M_{23}\oplus M_{34})=0$ by
Proposition~\ref{prop:reducible-counterexample}, and $\Hom(M_{22},P)=0$ because
the two modules have disjoint supports. Since $M_{12}\notin\Ssub$, the
subcategory $\Ssub$ is not extension closed in $\E^{\Isub}$. Thus the triple is
not a premutation triple.

Finally, $L$ is nonzero in $\Zsub/[\Isub]$. Maps between different components
vanish, and every composite
$$
  L\longrightarrow P\longrightarrow L
$$
is zero: the first map has image in the socle of $P$, whereas the second map
factors through the top of $P$. Thus $1_L$ does not factor through any object
of $\Isub$, and $\Zsub/[\Isub]\neq0$.
\end{proof}

\section{Relation with earlier quotient constructions}\label{sec:relations}
We emphasize that our characterization concerns the canonical pretriangulated structure associated with a premutation triple. The underlying additive category may admit another triangulated structure. Such a structure need not be related to the mutation functors $\Sigma$ and $\Omega$. Thus we do not claim that the existence of an arbitrary triangulated structure on the quotient implies $(\mathrm{MT4})$.

Our result is different from earlier criteria for classical quotient constructions. Zhou and Zhu \cite{ZhouZhu} prove that an extension closed subcategory $\mathcal A$ has a triangulated quotient $\mathcal A/\mathcal X$ if $(\mathcal A,\mathcal A)$ is an $\mathcal X$ mutation pair. Under Auslander--Reiten assumptions, they also prove a converse. Zheng, Cao and Wei \cite{ZhengCaoWei} construct a pretriangulated subfactor under one sided closure assumptions. When $\E(\mathcal X,\mathcal X)=0$, they show that triangulatedness implies the mutation pair and extension closure conditions. Our setting is different. We start with three subcategories $(\Ssub,\Zsub,\Vsub)$. We characterize when Iitsuka's canonical pretriangulated structure is triangulated. The criterion is given by the compatibility of the functors  $(\,\cdot\,)^{-}$ and $(\,\cdot\,)^{+}$.

Iitsuka \cite{Iitsuka} verifies $(\mathrm{MT4})$ in the Frobenius, rigid mutation and orthogonal mutation examples. For concentric twin cotorsion pairs, Theorem~\ref{thm:tcp-application} shows that the two quotient inclusions in Definition~\ref{def:tcp-inclusions} imply $(\mathrm{MT4})$. These inclusions are equivalent to $(\mathrm{MT4}^{-})$. For reducible triples, Theorem~\ref{thm:reducible-application} shows that extension closure is the only remaining obstruction. Proposition~\ref{prop:reducible-counterexample} and Example~\ref{cor:nonzero-reducible-counterexample} exhibit this obstruction for both zero and nonzero quotients.

Mochizuki and Ogawa \cite{MochizukiOgawa} study enhanced extriangulated subquotients. Huerta \cite{Huerta} studies Iyama--Yoshino reduction in extriangulated categories. These approaches do not give the converse criterion for premutation triples proved here.

We construct the mutation functors from four operations,
$$
  \langle1\rangle,\hspace{2mm} \sigma,
  \hspace{2mm}
  \langle-1\rangle,\hspace{2mm} \omega.
$$
In $(\mathrm{MT4})$, we require the two subcategories $\ul{\Zsub\langle1\rangle}^{-}$ and $\ul{\Zsub\langle-1\rangle}^{+}$ to be equal. In Theorem~\ref{thm:main}, we prove that this equality is equivalent to invertibility of the unit and the counit. Thus $(\mathrm{MT4})$ is exactly the obstruction to triangulatedness of the canonical mutation structure.

\vspace{6mm}

\hspace{-4mm}\textbf{Data Availability}\hspace{2mm} Data sharing not applicable to this article as no datasets were generated or analysed during
the current study.
\vspace{2mm}

\hspace{-4mm}\textbf{Conflict of Interests}\hspace{2mm} The authors declare that they have no conflicts of interest to this work.

\end{document}